\documentclass[12pt,twoside,final,psamsfonts]{amsart}
\usepackage[psamsfonts]{amssymb}
\usepackage[dvips,final]{graphicx,psfrag}
\usepackage{times,a4wide}
\usepackage{esint}

\theoremstyle{plain}
\newtheorem*{theorem*}{Theorem}
\newtheorem*{theoremA}{Theorem A}
\newtheorem*{theoremB}{Theorem B}
\newtheorem*{theoremC}{Theorem C}
\newtheorem{theorem}{Theorem}[section]
\newtheorem{proposition}[theorem]{Proposition}
\newtheorem*{proposition*}{Proposition}

\newtheorem*{corollary*}{Corollary}
\newtheorem{lemma}[theorem]{Lemma}
\newtheorem*{lemma*}{Lemma}

\theoremstyle{remark}
\theoremstyle{definition}

\newtheorem*{remark*}{Remark}

\theoremstyle{definition}

\newtheorem*{definition*}{Definition}

\begin{document}

%---------------------------------------------------------------------------
%-------Nouvelles commandes et definitions ---------------------------------
%---------------------------------------------------------------------------
\newcommand{\lil}{\lambda\in \Lambda}
\newcommand{\D}{\mathbb{D}}
\newcommand{\C}{\mathbb{C}}
\newcommand{\N}{\mathcal{N}}
\newcommand{\R}{\mathbb{R}}
\newcommand{\Z}{\mathbb{Z}}
\newcommand{\dist}{\operatorname{dist}}
\newcommand{\Int}{\operatorname{Int}}
\newcommand{\Hol}{\operatorname{Hol}}
\newcommand{\Har}{\operatorname{Har}}
\newcommand{\Harmd}{\Har_+(\D)}
\renewcommand{\Re}{\mbox{Re}}
\renewcommand{\Im}{\mbox{Im}}

\renewcommand{\qedsymbol}{$\blacksquare$}

\parskip 8 pt

\title{The Corona Property in Nevanlinna quotient algebras and Interpolating sequences}

\author{Xavier Massaneda, Artur Nicolau, Pascal J. Thomas }

\address{A. Nicolau: Universitat Aut\`onoma de Barcelona\\
Departament de Matem\`a\-tiques\\
Edifici C, 08193-Bellaterra\\ Catalonia}
\email{artur@mat.uab.cat}

\address{X. Massaneda: Universitat  de Barcelona\\
Departament de Matem\`a\-tiques i Inform\`atica\\
Gran Via 585, 08007-Bar\-ce\-lo\-na\\ Catalonia}
\email{xavier.massaneda@ub.edu}

\address{P. Thomas:  
Institut de Mathématiques de Toulouse ; UMR5219 \\
Universit\'e de Toulouse ; CNRS \\
UPS IMT, F-31062 Toulouse Cedex 9 \\
 \\ France}
 \email{pascal.thomas@math.univ-toulouse.fr}


\thanks{First and second authors supported by the Generalitat de Catalunya (grants 2017 SGR 359 and 2017 SGR 395) and the Spanish Ministerio de Ciencia e Innovaci\'on (projects MTM2014-51834-P, MTM2017-83499-P and  MTM2014-51824-P, MTM2017-85666-P.)}

\date{\today}

%\keywords{Nevanlinna class, Corona problem, quotient algebra, Interpolating sequence }

\begin{abstract} 
Let $I$ be an inner function in the unit disk $\D$ and let $\N$ denote the Nevanlinna class. We prove that under natural assumptions, Bezout equations in the quotient algebra $\N/I\N$ can be solved if and only if the zeros of $I$ form  a finite union of Nevanlinna interpolating sequences. This is in contrast with the situation in the algebra of bounded analytic functions, where being a
finite union of  interpolating sequences is a sufficient but not necessary condition. An analogous result in the Smirnov class is proved as well as several equivalent descriptions of Blaschke products whose zeros form a finite union of interpolating sequences in the Nevanlinna class. 
\end{abstract}

\maketitle

An analytic function $f$ in the unit disk $\mathbb{D}$ is in the Nevanlinna class $\N$ if the function $\log^+|f| $ has a harmonic majorant in $\mathbb{D}$. The Nevanlinna class is
an algebra where any nonvanishing function is invertible, and the largest function space where the classical inner-outer factorization holds. In particular, if $f\in \N$, then $f$ factors as $f=BF $, where $B$ is a Blaschke product and $F=E S_1/S_2$, 
where $E, S_1, S_2$ are nonvanishing analytic functions, $E$
is outer and $S_1, S_2$ are singular inner. 

% In this paper we study the analogue for $\N$ of the corona problem in quotient algebras studied in \cite{GMN}. Any principal ideal of $\N$ will be of the form $B \N$, with $B$ some Blaschke product (and those are the only closed % ideals of $\N$ \cite[Theorem 9]{Ma}).
% We show that the natural corona problem in  $\N_B:=\N/B\N$ is solvable if and only if the zeros of $B$ are a finite union of interpolating sequences for the Nevanlinna class (Theorem A).

%begHC
Carleson's classical Corona Theorem \cite{Ca} for the algebra $H^\infty$ of bounded analytic functions in the unit disk  states that given $f_1,\ldots, f_n\in H^\infty$ there exist $g_1,\ldots, g_n\in H^\infty$ such that the B\'ezout equation $f_1g_1+\cdots+f_ng_n\equiv1$ holds if and only if 
there exists $\delta>0$ such that for all $z\in \D$,
\begin{equation}
\label{HCor}
|f_1(z)|+\cdots+|f_n(z)|\geq \delta.
\end{equation}
\noindent The Corona problem for the  Nevanlinna class was solved by R. Mortini in \cite[Satz 4]{M} (see also \cite{Ma}).
Let $\Harmd$ denote the cone of positive functions in the unit disk.  Given $f_1,\ldots,f_n\in \N$, the B\'ezout equation $f_1g_1+\cdots+f_ng_n\equiv1$ can be solved with functions $g_1,\ldots,g_n\in \N$ if and only if there exists $H\in\Harmd$ such that
\begin{equation}
\label{NevCor}
|f_1(z)|+\cdots+|f_n(z)|\geq e^{-H(z)},\hspace{1cm} z\in\mathbb{D}.
\end{equation}
%endNC

%begHN
Notice that \eqref{NevCor} can be seen as a modification of \eqref{HCor}, where the constant bound from below is replaced by a control given
by a positive harmonic majorant (or minorant). 
%The same phenomenon occurs with the corona theorems for $H^\infty$ and $\N$. 
It turns out that we can use this as a general guiding principle:  
the known results for $H^\infty$ are valid as well for $\N$, once we replace uniform bounds by positive harmonic functions, in a proper way. We shall recall how this is applied to interpolating sequences, but first 
let us turn to the Corona results for quotient algebras of $H^\infty$.

%begHQ1
A function $I\in H^\infty$ is called inner if $\lim_{r\to1} |I(r\xi)|=1$ for almost every $\xi\in\partial\mathbb{D}$. Any principal ideal of $H^\infty$ is of the form $I H^\infty$.
P. Gorkin, R. Mortini and N. Nikolskii  \cite{GMN} considered the quotient algebra $H^\infty /IH^\infty$ and its \emph{visible spectrum,} which consists of $\Lambda=Z(I)$, the zeros of $I$.  We say  that $H^\infty /IH^\infty$ has the Corona Property if: 
given $f_1,\ldots, f_n\in H^\infty$ such that 
\begin{equation*}
 |f_1(\lambda)|+\cdots+|f_n(\lambda)|\geq \delta \qquad \lambda\in\Lambda,
\end{equation*}
for some $\delta>0$, 
then there exist $g_1,\ldots, g_n, h\in H^\infty$ such that $f_1g_1+\cdots+f_ng_n=1 +h I$.

Among many other things, it is proved in  \cite{GMN} that $H^\infty /IH^\infty$ has the Corona Property
 if and only if $I$ satisfies the \emph{weak embedding property} (WEP): 
\begin{equation}
\label{wepcond}
\mbox{For any }\epsilon>0,\mbox{ there exists }\eta>0\mbox{ such that }|I(z)|\geq \eta 
\mbox{ for }  z\mbox{ with }\rho(z,\Lambda)>\epsilon.
\end{equation}
%endHQ1

%begNQC
%Since any inner function can be factored $I=BF$, where $B$ is a Blaschke product and $F$ is invertible in $\N$, one has $\N_I=\N_B$. 
In the framework of the algebra $\N$, since any nonvanishing function is invertible, it is enough to consider
a Blaschke product $B$ with zero set $\Lambda$ and the quotient algebra $\N_B = \N / B \N$, 
the elements of which are in one-to-one correspondence with their  traces over the zeros of $B$. 

\begin{definition*}
We say that the Corona Property holds for $\N_B$ if for any positive integer $n$ and any  $f_1,\ldots,f_n\in \N$ for which there exists $H\in \Harmd$ such that 
\begin{equation}\label{1}
|f_1(\lambda)|+\cdots+|f_n(\lambda)|\geq e^{-H(\lambda)}\qquad \lambda\in\Lambda,
\end{equation}
there exist $g_1,\ldots,g_n , h \in \N$ such that $f_1g_1+\cdots+f_ng_n  = 1 +  B h $, that is, there exist $g_1,\ldots,g_n \in \N$ such that 
\[
 f_1(\lambda)g_1(\lambda)+\cdots+f_n(\lambda)g_n(\lambda)=1, \quad \lambda\in\Lambda. 
\]
\end{definition*}

Observe that condition \eqref{1} is necessary and that the case $n=1$ 
simply expresses invertibility in $\N_B$. 
%endNQC

Examples of inner functions satisfying the WEP \eqref{wepcond} are provided by interpolating
Blaschke products, that is to say, Blaschke products associated to an
interpolating sequence. Recall that a sequence $\Lambda=\{\lambda_k\}_k$ of points in $\mathbb{D}$ is called interpolating (for $H^\infty$) if for any bounded sequence $\{w_k\}_k$ of complex numbers there exists $f\in H^\infty$ with $f(\lambda_k)=w_k$, $k\geq 1$.  When $B$ denotes the Blaschke product associated to a Blaschke sequence $\Lambda$ we write, for $\lambda\in\Lambda$,
\[
 b_\lambda(z)=\frac{z-\lambda}{1-\bar\lambda z}\quad\textrm{and}\quad B_\lambda(z)=\frac{B(z)}{b_\lambda(z)}\ .
\]

In \cite{Ca2}, L. Carleson proved that $\Lambda$ is interpolating if and only if
\begin{equation}
\label{condCarl}
\inf_{k\geq 1} |B_{\lambda_k}(\lambda_k)|=(1-|\lambda_k|^2)|B'(\lambda_k)|=\prod_{m\neq k}\left|\frac{\lambda_m-\lambda_k}{1-\bar\lambda_k\lambda_m}\right|>0.
\end{equation}
%endHI

%begHQ2
%It is easy to see that any finite product of interpolating Blaschke products has the WEP. 
In the $H^\infty$ case, there currently is no known  characterization of
the inner functions satisfying the weak embedding property \eqref{wepcond}. 
However, easy examples are provided by finite unions of interpolating sequences (for $H^\infty$),
sometimes known as Carleson-Newman sequences.  On the other hand in \cite{GMN} there is an example, constructed in collaboration with S. Treil and V. Vasyunin, of a WEP Blaschke product which is not a finite product of interpolating Blaschke products. The papers \cite{Bo1}, \cite{BNT} and \cite{NV}  contain further contributions in this direction.
%endHQ2

%begNI
% We now turn to the Nevanlinna analogue of ($H^\infty$-)interpolating sequences. 
% Since the trace of a Nevanlinna function on a sequence $\Lambda$ is not so easy to characterize,
% one way to define them is to require that the trace space $\N|_\Lambda$ only depends on the 
% moduli of the values at each $\lambda \in \Lambda$.
% Such sequences were characterized in \cite[Theorem 1.2]{HMNT}, where a more
% explicit description of the trace space was also given.

We now turn to the Nevanlinna analogue of ($H^\infty$)-interpolating sequences. Given a sequence $\Lambda=\{\lambda_k\}_k$ in $\mathbb{D}$, let $W(\Lambda)$ be the set of sequences $\{w_k\}_k$ of complex numbers such that the map $\lambda_k\longmapsto \log^+|w_k|$, $k\geq 1$, has a positive harmonic majorant in $\mathbb{D}$. Observe that if $f\in \N$, then the sequence $\{f(\lambda_k)\}_k$ is in $W(\Lambda)$. 

\begin{definition*}
A sequence of points $\Lambda=\{\lambda_k\}_k\subset\mathbb{D}$ is called \emph{interpolating for} $\N$ if for any $\{w_k\}_k\in W(\Lambda)$ there exists $f\in \N$ such that $f(\lambda_k)=w_k$, $k\geq 1$.  
\end{definition*}

The main theorem in \cite{HMNT} states that
 $\Lambda=\{\lambda_k\}_{k\geq 1}$ is an interpolating sequence for $\N$ if and only if there exists $H\in \Harmd$ such that
\begin{equation}\label{intN}
\left|B_{\lambda_k}(\lambda_k)\right|=(1-|\lambda_k|^2)|B'(\lambda_k)|\geq e^{-H(\lambda_k)},\hspace{1cm} k\geq 1.
\end{equation}
In keeping with our general principle, this is a harmonic majorant version of 
\eqref{condCarl}.

A Blaschke product with zeros $\Lambda$ is called a \emph{Nevanlinna Interpolating Blaschke Product} (NIBP) if $\Lambda$ is an interpolating sequence for $\N$.  Other properties and characterizations of Nevanlinna interpolating sequences have been given recently in \cite{HMN1}.
%En aquesta equació vols que et posi := o un = amb tres línies (\equiv)? A la següent línea he afegit unes cometes que crec que ajuden.
Consider the pseudohyperbolic distance in $\D$, defined as
\[
 \rho(z,w)=\left|b_w(z)\right| =\left|\frac{z-w}{1-\bar z w}\right|\ ,
\]
and the corresponding pseudohyperbolic disks $D(z,r)=\{w\in\D : \rho(z,w)<r\}$. Given a sequence $\Lambda$ and a point $z\in \D$ let $\rho(z,\Lambda)=\inf_{\lambda\in\Lambda} \rho(z,\lambda)$. According to \cite[Theorem 1.2]{HMN1} a Blaschke product $B$ with simple zeros $\Lambda$ is a NIBP if and only if there exists $H\in \Harmd$ such that
\begin{equation}\label{intN2}
\left|B(z)\right|\geq e^{-H(z)}\rho(z,\Lambda),\hspace{1cm} z\in\mathbb{D}.
\end{equation}
%endNI

%begHN2
It was proved by Vasyunin in \cite{Vas78} (see also \cite{K-L}) that $B$ is an interpolating Blaschke product (for $H^\infty$) if and only if there exists a constant $\delta>0$ such that
\begin{equation*}
|B(z)|\geq\delta\rho(z,\Lambda),\hspace{1cm}z\in\mathbb{D}.
\end{equation*}
Again in the framework of our general principle, condition \eqref{intN2} can be seen as the harmonic majorant version of Vasyunin's result.  In contrast with the situation in $H^\infty$, our main result states that the Corona Property holds for $\N_B$ if and only if $B$ is a finite product of NIBP's. In other words, in the context of $\N$, there are no non-trivial WEP inner functions. 

\begin{theoremA}\label{main}
Let $B$ be a Blaschke product and let $\Lambda$ be its zero sequence. The following conditions are equivalent:
\begin{itemize}
\item [(a)] The Corona Property holds for $\N_B$.
\item [(b)] For any $H_1\in \Harmd$, there exists $H_2\in \Harmd$ such that $|B(z)|\geq e^{-H_2(z)}$ for any $z\in\mathbb{D}$ such that $\rho(z,\Lambda)\geq e^{-H_1(z)}$.
\item [(c)] $B$ is a finite product of Nevanlinna interpolating Blaschke products.
\end{itemize}
\end{theoremA}

By analogy with \eqref{wepcond} in the $H^\infty$-case, a Blaschke	product $B$ satisfying (b) will be called \emph{Nevanlinna WEP}. 
It can be seen from our main result, but will be proved beforehand as Lemma \ref{subproduct}
that any subproduct of a Nevanlinna WEP Blaschke product is also Nevanlinna WEP while in the setting of $H^\infty$ subproducts of WEP Blaschke products may fail to be WEP. 

The equivalence between (a) and (b) is proved along the same lines as in the case of $H^\infty$ (see Section 3). 
The main difficulty is to show that (b) implies (c) which is the implication that does not hold in the context of $H^\infty$.  The following auxiliary result about harmonic majorants may be of independent interest:

\begin{theoremB}\label{critical-growth}
 Let $\{\lambda_k\}_k\subset\D$ and let $m_k,M_k>0$, $k\geq 1$ be such that
 \begin{itemize}
  \item [i)] $\lim\limits_{k\to\infty} m_k=+\infty$,
  \item [ii)] $\limsup\limits_{k\to\infty}\, (1-|\lambda_k|) M_k<+\infty$.
 \end{itemize}

 Assume that the map $\lambda_k \longmapsto M_k$, $k\geq 1$ has no harmonic majorant, i.e., there is no $H\in\Harmd$ such that $M_k\leq H(\lambda_k)$, $k\geq 1$. 
 Then there exist a subsequence of positive integers $\{k_i\}_i$ and $H\in\Harmd$ such that
 \begin{itemize}
  \item [(a)] $\dfrac{M_{k_i}}{m_{k_i}}\leq H(\lambda_{k_i})\leq M_{k_i}$, $i\geq 1$, 
  \item [(b)] The map $\lambda_{k_i} \longmapsto M_{k_i}$, $i\geq 1$, has no harmonic majorant.
 \end{itemize}
 Consequently, the map $\lambda_{k_i} \longmapsto m_{k_i} H(\lambda_{k_i})$, $i\geq 1$ has no harmonic majorant either.
\end{theoremB}

Our last result collects several different descriptions of finite products of NIBP's. Analogous results for interpolating Blaschke products were proved by Kerr-Lawson \cite{K-L}, Gorkin \& Mortini \cite{GM} and Borichev \cite{Bo1}.

Given $H\in \Harmd$ and $z\in\mathbb{D}$, denote $D_H(z)=D(z,e^{-H(z)}) $. Given a Blaschke product $B$ and $z \in \D$, let $|B(N)(z)|$ denote the value at the point $z\in\mathbb{D}$ of the modulus of the Blaschke product obtained from $B$ after deleting the $N$ zeros of $B$ which are closest (in the pseudo-hyperbolic metric) to $z$.

\begin{theoremC}\label{descriptions}
Let $B$ be a Blaschke product with zeros $\Lambda$ and let $N$ be a positive integer. The following conditions are equivalent:
\begin{itemize}
\item  [(a)] $B$ is a product of $N$ Nevanlinna interpolating Blaschke products.
\item  [(b)]  There exists $H_1\in \Harmd$ such that 
\begin{equation*}
|B(z)|\geq e^{-H_1(z)}\rho^N\left(z,\Lambda\right),\hspace{1cm}z\in\mathbb{D}.
\end{equation*}
\item  [(c)]  % Let $B(N)(z)$ denote the value at the point $z\in\mathbb{D}$ of the Blaschke product obtained from $B$ after deleting the $N$ zeros of $B$ which are closer (in the pseudo-hyperbolic metric) to $z$. 
There exists $H_2\in \Harmd$ such that $|B(N)(z)|\geq e^{-H_2(z)}$, $z\in\mathbb{D}$.

\item [(d)] There exists $H_3\in \Harmd$ such that
\begin{equation*}
D_N(B)(z)=\sum_{j=0}^{N}(1-|z|)^j|B^{(j)}(z)|\geq e^{-H_3(z)},\hspace{1cm}z\in\mathbb{D}.
\end{equation*}

\end{itemize}
\end{theoremC}

The equivalence between (a), (b) and (d) for $N=1$ was proved in \cite[Theorem 1.2]{HMN1}. On the other hand (b) can be seen as the Nevanlinna version of 
\[
 |B(z)|\geq \delta \rho^N\left(z,\Lambda\right),\hspace{1cm}z\in\mathbb{D},
\]
which is known to be equivalent to $B$ being a product of $N$ interpolating Blaschke products (see \cite[Theorem 3.6]{GM} or \cite[Proposition 1]{Bo1}).

A consequence of these characterizations is the following result.

\begin{corollary*}
Let $B$ be a finite product of Nevanlinna interpolating Blaschke products. Then, there exists $H_0=H_0(B)\in \Harmd$  such that for any $g\in H^\infty$ with $|g(z)|\leq e^{-H_0(z)}$, $z\in\mathbb{D}$, the function $B-g$ factors as $B-g=B_1G$, where $B_1$ is a finite product of Nevanlinna interpolating Blaschke products and $G\in H^\infty$ is such that $1/G\in H^\infty$.
\end{corollary*}

The paper is structured as follows. In the first section we describe the setup and gather several auxiliary results to be used in the main proofs. Section 2 is devoted to prove the equivalence between (a) and (b) in Theorem A, which, as mentioned before, follows the scheme of the analogue for the $H^\infty$-case (see \cite[Theorem 3.3]{GMN}). Section 3 contains the proof of Theorem B, which is essential in the proof of the main implication given in Section 4. We give the statements of the analogous results for the Smirnov class in 
Section \ref{smirnov},
and how to adapt the proofs to this case. 
Finally, Section \ref{proofc} gives the proof of Theorem C and the Corollary.

We would like to thank our colleague Alexander Borichev for fruitful discussions 
on the topic of this paper.

\section{Preliminaries}

The Nevanlinna class consists of the holomorphic functions $f$ in the unit disk $\D$ for which $\log^+|f|$ has a positive harmonic majorant.  
Recall that any $H\in \Har_+(\D)$ is the Poisson integral of a positive measure $\mu$ on the unit circle, that is 
\[
 H(z)=P[\mu](z)=\int_{\partial\D} P(z,\zeta) d\mu(\zeta),
\]
where
\[
 P(z,\zeta)=\Re\left(\frac{\zeta+z}{\zeta-z}\right)=\frac{1-|z|^2}{|\zeta-z|^2}
\]
is the Poisson kernel in $\D$.

It is well-know that a holomorphic function $f$ in $\D$ is in $\N$ if and only if
\[
\lim_{r \to 1}\frac{1}{2\pi}\int_{0}^{2\pi} \log^+|f(re^{i\theta})|\;d\theta<\infty .
\]
Then $f$ admits non-tangential limits $f^*$ at almost every point of the circle. 
It is also a standard fact that  $f\in \N$ can be factored as
$
 f=B S_1 E / S_2,
$
where $B$ is a Blaschke product containing the zeros of $f$, $S_1, S_2$ are singular inner functions and $E$ is the outer function
\[
 E(z)= C\exp\left\{\int_{\partial\D} \frac{\zeta+z}{\zeta-z} \log|f^*(\zeta)| d\sigma(\zeta)\right\},
\]
where $|C|=1$. In particular
\[
 \log |E(z)|= P[\log|f^*|](z),\quad z\in\D\ .
\]
A function $S$ is singular inner if there exists a positive measure $\mu$ on $\partial\D$ singular with respect to the Lebesgue measure such that
\[
 S(z)=\exp\left\{- \int_{\partial\D} \frac{\zeta+z}{\zeta-z} d\mu(\zeta) \right\},\quad z\in\D\ .
\]

Throughout the proofs we will use repeatedly the well-known \emph{Harnack inequalities}: for $H\in\Har_+(\mathbb D)$ and $z,w\in\mathbb \D$,
\[
 \frac{1-\rho(z,w)}{1+\rho(z,w)}\leq\frac{H(z)}{H(w)}\leq\frac{1+\rho(z,w)}{1-\rho(z,w)}\ .
\]
We shall always assume, without loss of generality, that $H\in\Har_+(\mathbb D)$ is big enough so that for  $z\in D_H(\lambda)$ the inequalities
$1/2\leq H(z)/H(\lambda)\leq 2$ hold. Actually it is sufficient to assume $\inf\{H(z): z \in \D \} \geq \log 3$.

Our first auxiliary result states that a Blaschke sequence is in a sense sparse in $D_H(z)$, if $H$ is appropriately chosen. Consider the usual partition of $\mathbb D$ into the dyadic (Whitney) squares
\[
 Q_{k,j}=\bigl\{z=re^{i\theta}\in\mathbb D : 1-2^{-k}\leq r< 1-2^{-k-1}\ ,\ j\frac{2\pi}{2^k}\leq \theta<(j+1)\frac{2\pi}{2^k}\bigr\},
\]
where $k\geq 0$ and $ j=0,\dots 2^k-1$. Consider also the corresponding projections on $\mathbb T$:
\[
 I_{k,j}=\bigl\{e^{i\theta}\in\mathbb T : \ j\frac{2\pi}{2^k}\leq \theta<(j+1)\frac{2\pi}{2^k}\bigr\}\ .
\]
Given $z\in \D$ denote by $Q_z$ the dyadic region $Q_{n,k}$ such that $z\in Q_{n,k}$.

As explained in the Introduction, in the context of $H^\infty$ there exist WEP Blaschke products which are not finite products of interpolating Blaschke products. The construction of such examples is based on the fact that there are Blaschke sequences whose hyperbolic neighbourhood is large. More concretely, there exist a Blaschke sequence $\Lambda$ such that for any $\varepsilon >0$, the set $\{z \in \D : \rho(z, \Lambda) < \varepsilon \}$ contains a horodisc. Next auxiliary result says that in the context of the Nevanlinna class such examples do not exist. 

\begin{lemma}\label{far}
 Let $\Lambda$ be a Blaschke sequence. For each Whitney square $Q$ let $N(Q):=\# \Lambda\cap Q$. Define $H_\Lambda=P[\psi_\Lambda]$, where
 \[
  \psi_\Lambda=\sum_{n,k} \log N(Q_{n,k})\, \chi_{I_{n,k}}\ .
 \]
 There exists $c_0>0$ such that for all $H\in\Harmd$ such that $H\geq c_0 H_\Lambda$ the following property holds: 
 for all $z\in\D$ there exists $ \tilde z\in D_H(z)$ such that $\rho(\tilde z, \Lambda)\geq e^{-10 H(z)}$.

\end{lemma}

\begin{proof}
 There is no restriction in assuming that $z$ is such that $\Lambda\cap D_H(z)\neq\emptyset$. Pick $\lambda_k\in \Lambda\cap D_H(z)$. An area estimate shows that  $D_H(z)$ contains approximately $e^{18 H(z)}$ pairwise disjoint disks $D_j$ or radius $e^{-10 H(z)}$:
 \[
  \frac{\textrm{Area}(D_H(z))}{\textrm{Area}(D_j)}\simeq \frac{e^{-2 H(z)}}{e^{-20 H(z)}}=e^{-18 H(z)}.
 \]
Since $\# \Lambda\cap D_H(\lambda_k)\leq N(Q_{\lambda_k})$, by the definition of $H_\Lambda$ we deduce that there exists a universal constant $C>0$ 
% (universal, depending only on the definition of $\psi$) 
such that
 \[
  \#\{j\, :\, \Lambda\cap D_j\neq \emptyset \}\leq N(Q_{\lambda_k}) \leq e^{CH_\Lambda(\lambda_k)}.
 \]
This shows that there exists $D_j$ which contains no point of $\Lambda$; we may take as $\tilde z$ the centre of such $D_j$.
\end{proof}

The next result will be used in the proof of the main implication (b)$\Rightarrow$(c) in Theorem A.

\begin{lemma}\label{subproduct}
 Any subproduct of a Nevanlinna WEP Blaschke product is also a Nevanlinna WEP Blaschke product.
\end{lemma}

\begin{proof}
 Assume $B=B_1B_2$ is a Nevanlinna WEP Blaschke product and let $\Lambda$ its zero sequence. Denote also by $\Lambda_i$, $i=1,2$, the zero sequences of $B_i$, $i=1,2$. We want to show that $B_1$ is also a Nevanlinna WEP Blaschke product.
 Let $H_1\in\Harmd$ be such that $\rho(z,\Lambda_1)\geq e^{-H_1(z)}$. There is no restriction in assuming that $H_1$ satisfies also the conditions of Lemma~\ref{far}. Since $B$ is Nevanlinna WEP, if $z$ is such that moreover $\rho(z,\Lambda_2)\geq e^{-10 H_1(z)}$,
 there exists $H_2\in\Harmd$ such that  
 \[
  |B_1(z)|\geq |B(z)|\geq e^{-H_2(z)}.
 \]
 In case $\rho(z,\Lambda_2)\leq e^{-10 H_1(z)}$, by Lemma~\ref{far} we can pick $\tilde z\in\D$ with $\rho(\tilde z,z)\leq e^{-10 H_1(z)}$ and $\rho(\tilde z,\Lambda_2)\geq e^{-100 H_1(z)}$. Hence $\rho(\tilde z,\Lambda)\geq e^{-100 H_1(z)}$ and since $B$ is a Nevanlinna WEP Blaschke product, there exists $H_3\in\Harmd$ such that
 \[
  |B_1(\tilde z)|\geq |B(\tilde z)|\geq e^{-H_3(\tilde z)}.
 \]
 Since $B_1$ has no zeros in $D(z,e^{-5 H_1(z)})$, Harnack's inequalities applied in that disc give
 \[
 |B_1( z)|\geq e^{-2 H_3(z)}.
 \]
\end{proof}

\section{Proof of Theorem A $(a)\Leftrightarrow(b)$}

(a)$\Rightarrow$(b). Assume that given $f\in\N$ there exist $g\in \N$  such that $ f  g-1\in  B\N$ if and only if there exists $H\in\Harmd$ such that $ |f(\lambda)|\geq e^{-H(\lambda)}$, $\lambda\in\Lambda$. For any positive harmonic function $H_1$ we need to show that $ - \log |B|$
admits a harmonic majorant on the set
\[
\left\{ z\in \D: \rho(z,\Lambda) \ge e^{-H_1(z)} \right\}.
\]
Increasing the function $H_1$, we only make this set larger, so we can use Lemma \ref{far}
and replace $H_1$ by another positive harmonic function denoted again by $H_1$ so that for any
$z\in \D$, there exists $w$ such that $\rho(z,w)\le \frac12$ and $\rho(w,\Lambda) \ge e^{-H_1(z)}$. Now for each Whitney cube $Q_{k,j}$ so that 
\[ 
\overline Q_{k,j} \cap \left\{ z\in \D: \rho(z,\Lambda) \le \frac12 \right\} \neq \emptyset,
\]
choose a point $a_{k,j}$ such that 
\[
\log \frac1{|B(a_{k,j})|} = \max \left\{ \log \frac1{|B(z)|}: z \in \overline Q_{k,j},
\rho(z,\Lambda) \ge e^{-H_1(z)} \right\}.
\]
Let $A$ be the sequence made up of all the $a_{k,j}$. Since each closed Whitney cube intersects exactly $8$ other cubes, a standard combinatorial lemma shows that
 $A$ is the union of  $8$ disjoint sequences $A_j$  so that each $A_j$ never contains two points in neighboring 
 cubes, and therefore is $\rho$-separated with an absolute constant $r_0>0$.  It is easy to see that 
 each $A_j$, being separated and in a fixed hyperbolic neighborhood of $\Lambda$,
is again a Blaschke sequence. So, by \cite[Corollary 1.9]{HMNT}, each $A_j$
 is a Nevanlinna interpolating sequence, and even a Smirnov interpolating sequence. Let $B_A$ be the Blaschke product with zero set $A$,
 then by \eqref{intN2} 
 there exists a harmonic function $H_A>0$ such that 
 \[ 
 \log |B_A(z)| \ge  -H_A(z) + 8 \log \rho (z,A).
 \]
 In particular, for $z \in \Lambda$, $\log |B_A(z)| \ge  -H_A(z) - 8 H_1 (z)$. So property (a)
 implies that there exist  functions $g,g_A \in \N$ so that $gB+g_AB_A=1$. In particular, 
 on the set $A$, $B(a_{k,j})= g(a_{k,j})^{-1}$, so applying the canonical factorization to $g$,
 we see that $a_{k,j} \mapsto  - \log |B(a_{k,j})|$ admits a positive harmonic majorant, 
 which we denote by $H_2$. 
 
% There exists $r_1>0$ such that 
The set
 \[
 \left\{ z\in \D: \rho(z,\Lambda) \le \frac12 \right\}
 \]
 is covered by the union of the $Q_{k,j}$
 %hyperbolic discs $D(a_{k,j},r_1)$ 
 for $a_{k,j} \in A$. Since each of those cubes is of bounded hyperbolic diameter, Harnack's inequality
 implies that there exists a constant $C_1$ such that for all $z\in Q_{k,j}$ such that
 $\rho(z,\Lambda) \ge e^{-H_1(z)}$,
 \[
 C_1 H_2(z) \ge H_2(a_{k,j}) \ge - \log |B(a_{k,j})| \ge - \log |B(z)| ;
 \]
on the other hand, when  $\rho(z,\Lambda) \ge 1/2$, by \cite[Proposition 4.1]{HMNT}, 
$- \log |B(z)|\le C P[\psi]$, where $\psi = \sum_{\lambda \in \Lambda} \chi_{I_\lambda}$,
with $I_\lambda$ the Privalov shadow of $\lambda$. Summing those two majorants, we do have a
positive harmonic majorant for $-\log |B|$ on the set where  $\rho(z,\Lambda) \ge e^{-H_1(z)}$. This finishes the proof. 
\bigskip

% (b)$\Rightarrow$(a). 

The proof of the converse implication uses the following elementary lemma.

\begin{lemma}
\label{schwarz}
 Let $f\in\N$ with $H_0\in\Harmd$ such that $|f(z)|\leq e^{H_0(z)}$, $z\in\D$. For any $\lambda\in\D$,
 \[
  |f(z)-f(\lambda)|\leq 6 \rho(z,\lambda)\, e^{2H_0(\lambda)}, \quad z\in D_{\log 3}(\lambda).
 \]
\end{lemma}

\begin{proof}
By Harnack's inequality, the function $g(z):=f(z) e^{-2 H_0(\lambda)}$ is bounded by $1$
on $D_{\log 3}(\lambda)$. Let $\varphi_\lambda$ denote the automorphism of $\D$ exchanging $\lambda$ and $z$ (explicitly $\varphi_\lambda(z)=(\lambda-z)/(1-\bar\lambda z)$). The map
 \[
  G(\xi):=\frac 12\left(g(\varphi_\lambda(\frac{\xi }3))-g(\lambda)\right)
 \]
 takes $\D$ into $\bar\D$ and $G(0)=0$, so by the Schwarz Lemma
 \[
  \left|g(\varphi_\lambda(\frac{\xi }3 ))-g(\lambda)\right|\leq 2 |\xi|, \quad \xi\in\D, 
 \]
 and therefore
 \[
  \left|f(\varphi_\lambda(\frac{\xi }3))-f(\lambda)\right|\leq 2 |\xi| e^{2 H_0(\lambda)} , \quad \xi\in\D. 
 \]
 Taking $z=\varphi_\lambda(\xi /3)$; $\xi=3\varphi_\lambda(z)$ we  finally get
 \[
  |f(z)-f(\lambda)|\leq 6 |\varphi_\lambda(z)| e^{2 H_0(\lambda)}\ .
 \]
\end{proof}

(b)$\Rightarrow$(a).  Let $B$ be a Blaschke product and $\Lambda$ its zero sequence. Let now $f_1,\dots,f_n\in\N$ and let $H_0, H\in\Harmd$ be such that $H \ge \log 2$ and
\begin{align*}
 & \sum_{j=1}^n |f_j(\lambda)|\geq e^{-H(\lambda)}, \quad \lambda\in\Lambda, \\
 & \sum_{j=1}^n |f_j(z)|\leq e^{H_0(z)}, \quad z\in\D.
\end{align*}

The previous lemma yields, for all $z\in D_{\log 3}(\lambda)$ and all $j=1,\dots, n$,
\[
 |f_j(z)|\geq |f_j(\lambda)|-6\rho(z,\lambda) e^{2 H_0(\lambda)},
\]
hence
\[
 \sum_{j=1}^n |f_j(z)|\geq e^{-H(\lambda)}-6n\rho(z,\lambda) e^{2 H_0(\lambda)}\ .
\]

If $z$ is such that $\rho(z,\lambda)<1/(6n) e^{-(2 H_0(\lambda)+2 H(\lambda))}$ the previous estimate and Harnack's inequality give
\[
 \sum_{j=1}^n |f_j(z)|\geq e^{-H(\lambda)}-e^{-2H(\lambda)}\geq e^{-2H(\lambda)}. %\geq e^{-4H(z)}.
\]

If $z$ is such that $\rho(z,\lambda)\geq 1/(6n) e^{-(2 H_0(\lambda)+2 H(\lambda))}$ we apply the hypothesis: there exists $H_2\in\Harmd$ such that
$|B(z)|\geq e^{-H_2(z)}$.

All combined, there exists $H_3\in\Harmd$ such that
\[
 \sum_{j=1}^n |f_j(z)| + |B(z)| > e^{-H_3(z)}, \quad z\in\D,
\]
and again by the Corona Theorem in $\N$ there exist $g_1,\dots g_n,G\in \N$ such that $\sum_j f_j g_j + BG=1$, as desired.

\section{Proof of Theorem B}

Without loss of generality we can assume that $(1-|\lambda_k|) M_k <1$ for any positive integer $k$. We shall construct recursively indices $k_j<k_j'<k_{j+1}$ and harmonic functions $H_j\in\Harmd$ which will 
lead to the definitions of $k_i$ and $H$. Pick $k_1 = 1$, $k_1' = 2$ and $H_1 = 1$. For $j>1$ assume thus that $H_1, \dots, H_{j-1}$ and $k_l<k_l'$, $l\leq j-1$ have been defined. Pick $k_j>k_{j-1}'$ such that
 \begin{equation}\label{mj}
  \forall k\geq k_j, \quad m_k >3^j.
 \end{equation}
 For $\kappa\geq k_j$ let 
 \[
  c_j(\kappa):=\inf\bigl\{ U(0)\, :\, U\in\Harmd, \ 3^j\sum_{l=1}^{j-1} H_l(\lambda_k)+ U(\lambda_k)\geq M_k\ \textrm{for}\ k_j\leq k\leq \kappa\bigr\}.
 \]
 Since $\lambda_k\longmapsto M_k$, $k\geq 1$, has no harmonic majorant $c_j(\kappa)$ increases to $\infty$ as $\kappa\to\infty$.  We define 
 \[
  k_j'=\min\{\kappa\geq k_j\ :\ c_j(\kappa)\geq 2^j\}
 \]
 Consider the Privalov shadow of $\lambda_n$:
\begin{equation}
\label{privdef}
 I_{_n}=\bigl\{ e^{i\theta}\in\partial\D : |\theta- \arg(\lambda_{n})|<\frac 12 (1-|\lambda_{n}|)\bigr\}.
\end{equation}
Observe that if $H$ is the Poisson integral of the characteristic function of $I_{\lambda_n}$, then there exists a universal constant $C>0$ such that $H(\lambda_n) \geq C (1-|\lambda_n|)^{-1} $. By assumption $H(\lambda_n) > C M_n$ and we deduce $c_j (k_j) < C$. Hence $k_j' > k_j$. Denote $c_j=c_j(k_j')$.
 
 \textit{Claim.}
For $j$ sufficiently large
\begin{equation}\label{ck-estimate}
 c_j\leq 2^{j+1}
\end{equation}

\emph{Proof}. The minimality of $k_j'$ provides $U_0\in\Harmd$ with $U_0(0)<2^j$ and such that,
\[
 3^j \sum_{l=1}^{j-1} H_l(\lambda_k)+ U_0(\lambda_k)\geq M_k\quad \textrm{for $k_j\leq k \leq k_j'-1$}.
\]
Let now $I:=I_{{k_j'}}$, as in \eqref{privdef} above,  and 
%Let $\delta$ denote the Dirac mass at $\lambda_{k_j'}/|\lambda_{k_j'}|$ (the centre of $I$) and 
define $G=P[2 M_{k_j'}\chi_I]$. Since by construction $G(\lambda_{k_j'})\geq M_{k_j'}$, the function $U_1:=U_0+G$ satisfies
\begin{equation}\label{jkjk}
 3^j\sum_{l=1}^{j-1} H_l(\lambda_k)+  U_1(\lambda_k)\geq M_k \quad \textrm{for}\ k_j\leq k\leq k_j',
\end{equation}
and therefore
\[
 c_j\leq U_0(0)+G(0)\leq 2^j+2 M_{k_j'}|I|\ .
\]
By the hypothesis ii) we have $M_{k_j'}|I|=M_{k_j'}(1-|\lambda_{k_j'}|) < 1$, which gives the Claim.

\bigskip

Let $H_j$ be the harmonic function giving the solution to the extremal problem
\[
 \tilde c_j:=\inf\bigl\{ H(0)\, :\, H\in\Harmd, \ \sum_{l=1}^{j-1} H_l(\lambda_k)+ H(\lambda_k)\geq M_k/m_k \ \textrm{for}\ k_j\leq k\leq k_j'\bigr\}.
\]
By \eqref{mj}, if $U\in\Harmd$ is such that \eqref{jkjk} holds then, for $k_j\leq k\leq k_j' $,
\[
 \sum_{l=1}^{j-1} H_l(\lambda_k)+ \frac 1{3^j} U(\lambda_k)\geq \frac{M_k}{3^j}\geq \frac{M_k}{m_k}.
\]
This shows that $\tilde c_j=H_j(0)\le U(0)/3^j \leq c_j/3^j$, hence, for $j$ large, $H_j(0)\leq 2^{j+1}/3^j$ and the series $H=\sum_j H_j$ defines a positive harmonic function.

By construction, for $j\geq 1$ and $ k_j\leq k\leq k_j'$
\begin{equation}\label{hzj}
 H(\lambda_k)\geq \frac{M_k}{m_k}.
\end{equation}

Observe also that by definition of $c_j$ above,
\[
 \inf\bigl\{ H(0)\, :\, H\in\Harmd,\ H(\lambda_k)\geq M_k\ \textrm{for}\ k_j\leq k\leq k_j'\bigr\}\geq c_j\geq 2^j,
\]
hence the mapping 
\begin{equation}\label{mapping}
\lambda_k\longmapsto M_k \qquad  \  k_j\leq k\leq k_j', \quad j\geq 1
\end{equation}
cannot have a harmonic majorant. 

Now let $J_j=\{k\in [k_j, k_j'] : H(\lambda_k)\leq M_k\}$ and consider the subsequence $\{k_i\}_i:=\cup_j J_j$. Notice that the map
$\lambda_k\longmapsto M_k$ for $ k\in \cup_j ([k_j, k_j']\setminus J_j)$
has a harmonic majorant (the function $H$). Therefore, by \eqref{mapping}, the map
\[
 \lambda_k\longmapsto M_k \quad k\in \cup_j  J_j
\]
has no harmonic majorant. This is (b) in the statement. The estimates in (a) follow from \eqref{hzj} and the definition of $\{k_i\}_i$.

\section{Proof of Theorem A $(b)\Rightarrow (c)$}

The main difficulty is to prove that there exists $H\in\Harmd$ such that $\Lambda$ can be split into a finite number of sequences $\Lambda_j$ for which the disks $\{D_H(\lambda)\}_{\lambda\in\Lambda_j}$ are pairwise disjoint. We shall say that the sequences $\Lambda_j$ are $H$-\emph{separated}. This separation together with the WEP property will show then that each $\Lambda_j$ is Nevanlinna interpolating.

Let us see now that there exists $H\in\Harmd$ such that $\Lambda$ can be split into a finite number of $H$-separated sequences.
Assume otherwise that for all $H\in\Harmd$ and all $N\geq 1$ the sequence $\Lambda$ is not the union of $N$ $H$-separated sequences. Define, for $k\geq 1$, the minimal pseudohyperbolic diametre of a cloud of $N$ points of $\Lambda$ containing $\lambda_k$, that is, 
\[
 \rho_N(\lambda_k)=\inf\bigl\{\sup\limits_{z,w\in E} \rho(z,w)\, :\, E\subset \Lambda,\, \# E=N,\, \lambda_k\in E\bigr\}.
\]
Let also
\[
 M_N(\lambda_k)=\log\rho_N^{-1}(\lambda_k), \quad (\rho_N(\lambda_k)=e^{-M_N(\lambda_k)})\ .
\]
Since we are assuming that $\Lambda$ is not the union of  a finite number of $H$-separated sequences, for any $N\geq 1$, the function
\[
 \lambda_k \longmapsto M_N(\lambda_k), \quad k\geq 1,
\]
has no harmonic majorant. Otherwise there would exist $H\in\Harmd$ such that   $e^{-H(\lambda_k)}\leq \rho_N(\lambda_k)$, $k\geq 1$, and therefore  $\# \Lambda\cap D_H(\lambda_k)\leq N$. This would allow to split $\Lambda$ into $N$ $H$- separated sequences (see \cite[Lemma 2.4]{HMN2}).

\bigskip

\emph{Claim.} For every sequence $m_k$ tending to $\infty$ as $|\lambda_k|\to 1^{-}$
\begin{equation}\label{limsupM}
\limsup_{k\to\infty} (1-|\lambda_k|) M_{m_k}(\lambda_k)=0.
\end{equation}
 
\emph{Proof.} 
Assume otherwise that there exists a subsequence, still denoted $\{\lambda_k\}_k$  for which $(1-|\lambda_k|) M_{m_k} (\lambda_k) > \delta >0$, that is, $\rho_{m_k} (\lambda_k) < e^{-\delta / (1-|\lambda_k|)}$. Then there exist $m_k$ zeros of $B$ at a distance from $\lambda_k$ smaller than 
$e^{-\delta / (1-|\lambda_k|)}$. Taking another subsequence if necessary, we can assume that $\sum_k1/m_k<\infty$. Then we can choose 
$c_k >0$ such that $\sum_k c_k < \infty$ but $\limsup_k c_k m_k =  \infty$. 
Let $I_k$ be the Privalov shadow of $\lambda_k$, defined as in \eqref{privdef}, and
$H$ be the Poisson integral of the function $\sum_k c_k \frac1{|I_k|} \chi_{I_k}$. We can assume that $H$ satisfies the assumptions of Lemma~\ref{far}, replacing $H$ by $H+H_\Lambda$ if necessary. Then there exist $z_k\in D_H(\lambda_k)$ with $\rho(z_k , \Lambda) > e^{-10 H(\lambda_k)}$. Since $B$ has $m_k$ zeros near $\lambda_k$ we deduce that 
\[
 |B(z_k)| < e^{-m_k H(\lambda_k)},
\]
and since by construction $H(\lambda_k) > c_k /(1-|\lambda_k|)$, we deduce that 
\[
|B(z_k)| < e^{-\frac{m_k  c_k}{1-|z_k|}}. 
\]
On the other hand, if $B$ satisfies the WEP condition there exists $H_1\in\Harmd$ such that $|B(z_k)| > e^{-H_1(z_k)}$ and therefore 
\[
H_1(z_k)\geq \frac{ m_k  c_k }{1-|z_k|}\ .
\]
Since $\limsup_k c_k m_k =  \infty$, this contradicts Harnack's inequality. This finishes the proof of the Claim. 
% $\Box$

Now apply Theorem B to the sequence $\{\lambda_k\}_k$ and the values $M_k:=M_{m_k}(\lambda_k)$. We get a subsequence $\{k_j\}_j$ and $H\in\Harmd$ such that:
\begin{itemize}
 \item [(a)] $\dfrac{M_{k_j}}{m_{k_j}}\leq H(\lambda_{k_j})\leq M_{k_j},\quad j\geq 1$,
 \item [(b)] The map $\lambda_{k_j}\longmapsto M_{k_j}$, $j\geq 1$, has no harmonic majorant.
\end{itemize}

Let $H_1:=H+H_\Lambda$. Lemma~\ref{far} gives in particular, points $z_{k_j}\in D_{H_1}(\lambda_{k_j})$ such that
$ \rho(z_{k_j},\Lambda)\geq e^{-10 H_1(\lambda_{k_j})}$.
Since $B$ is Nevanlinna WEP there exists $H_2\in\Harmd$ such that
\[
 |B(z_{k_j})|\geq e^{-H_2(z_{k_j})}\geq e^{-2H_2(\lambda_{k_j})}, \quad j\geq 1.
\]
On the other hand, by (a), $\rho_{m_{k_j}}(\lambda_{k_j})\leq e^{-H(\lambda_{k_j})}$, $j\geq 1$, and therefore $B$ has at least $m_{k_j}$ zeros in $D_H(\lambda_{k_j})$. This implies that
\[
  |B(z_{k_j})|\leq e^{- H(\lambda_{k_j}) m_{k_j}}.
\]
Combining both inequalities,
\[
 H(\lambda_{k_j})\, m_{k_j}\leq 2 H_2(\lambda_{k_j}), \quad j\geq 1,
\]
which contradicts (b) (by the first inequality of (a)). 

This finishes the proof that there exists $H\in\Harmd$ such that $\Lambda$ can be split into a finite union of $H$-separated sequences.

According to Lemma~\ref{subproduct} we shall be done as soon as we prove that an $H$-separated Nevanlinna WEP Blaschke product is a Nevanlinna interpolation Blaschke product. Let then $B$ be a Nevanlinna WEP Blaschke product with zero set $\Lambda$ such that for some $H\in\Harmd$ the disks $\{D_H(\lambda)\}_{\lambda\in\Lambda}$ are pairwise disjoint. Since $B$ is Nevanlinna WEP there exists $H_2\in\Harmd$ such that
\[
 |B(z)|\geq e^{-H_2(z)} \quad\textrm{for}\quad z\in\cup_{\lambda\in\Lambda} \partial D_H(\lambda).
\]
In particular,
\[
 |B_\lambda(z)|\geq e^{-H_2(z)} \quad\textrm{for}\quad z\in  \partial D_H(\lambda).
\]
Since $B_\lambda$ has no zeros in $D_H(\lambda)$ we can apply the maximum principle to the harmonic function $\log |B_\lambda|^{-1}$ to deduce that
\[
 |B_\lambda(\lambda)|\geq \min_{z\in  \partial D_H(\lambda)} e^{-H_2(z)} \geq e^{-2H_2(\lambda)}.
\]
By \eqref{intN}, this implies that $\Lambda$ is Nevanlinna interpolating.

\section{The case of the Smirnov Class}
\label{smirnov}

A \emph{quasi-bounded} harmonic function is the Poisson integral of a measure absolutely
continuous with respect to the Lebesgue measure on the circle. We denote $QB(\D)$ the space
of the quasi-bounded harmonic functions, and $QB_+(\D)$ the cone of those which are nonnegative.

An analytic function $f$ in the unit disk $\mathbb{D}$ is in the Smirnov class $\N^+$ if the function 
$\log^+|f|$ has a (positive) quasi-bounded harmonic majorant in $\mathbb{D}$. 
A function in the Nevanlinna class is Smirnov
if and only if its canonical factorization has no singular function in the denominator,
i.e.  $f$ factors as $f=BF $, where $B$ is a Blaschke product and $F=E S$, 
$E$ is outer and $S$ is singular inner. 
The  class $\N^+$ is
an algebra where the invertible  functions are exactly the outer functions. 
So any quotient
of the Smirnov class by a principal ideal (which are the only closed ideals \cite[Theorem 2]{RS})
can be represented by $N^+_I:=\N^+/I\N^+$, where $I$ is 
an inner function.

R. Mortini's result \cite[Satz 4]{M} is valid for a whole class of spaces which he calls of Nevanlinna-Smirnov type, and for the Smirnov class it states that given $f_1,\ldots,f_n\in \N^+$, the B\'ezout equation $f_1g_1+\cdots+f_ng_n\equiv1$ can be solved with functions $g_1,\ldots,g_n\in \N^+$ if and only if there exists $H\in QB_+(\D)$ such that \eqref{NevCor} holds. In the same way,  interpolating sequence for $\N^+$
are characterized in \cite[Theorem 1.3]{HMNT} by the existence of $H\in QB_+(\D)$ such that \eqref{intN} holds.

We can generalize most of our results from the Nevanlinna to the Smirnov class, replacing
$\Harmd$ by $QB_+(\D)$. 

\begin{proposition}\label{smirnovprop}
Let $B$ be a Blaschke product and let $\Lambda$ be its zero sequence. The following conditions are equivalent:
\begin{itemize}
\item [(a)] The Corona Theorem holds for $\N^+_B$.
\item [(b)] For any $H_1\in QB_+(\D)$, there exists $H_2\in QB_+(\D)$ such that $|B(z)|\geq e^{-H_2(z)}$ for any $z\in\mathbb{D}$ such that $\rho(z,\Lambda)\geq e^{-H_1(z)}$.
\item [(c)] $B$ is a finite product of Smirnov interpolating Blaschke products.
\end{itemize}
\end{proposition}

Theorem B, and all the steps in the proofs in the previous sections, will go through when one replaces 
each instance of $\N$ by $\N^+$, of ``harmonic majorant'' by ``quasi-bounded harmonic majorant'', 
and of $\Harmd$ by $QB_+(\D)$. We only need to observe that the functions $\psi$ in the proof
of the implication (a) $\Rightarrow$ (b), 
$\psi_\Lambda$ in Lemma
\ref{far}, $G$ defined before equation \eqref{jkjk}, and $H$ defined after equation \eqref{limsupM}
are all quasi-bounded by construction. Finally, one should note that the series $\sum H_j$
in the proof of Theorem B (before equation \eqref{hzj}) has only nonnegative quasi-bounded terms,
so the partial sums of the boundary values form a monotone sequence and by Fatou's Lemma converge to a nonnegative
integrable function, instead of an arbitrary positive measure, so $H= \sum H_j$ is quasi-bounded.

We say that an inner function $I$ satisfies the Smirnov WEP if and only if it verifies the
analogue of property (b) in the proposition above.  
Notice that if it does and $I= BS$, where $B$ is a Blaschke product
and $S$ a singular inner function, then since $|I|\le |B|$ and $I^{-1}\{0\}=B^{-1}\{0\}$, it is clear that  $B$ must be Smirnov WEP as well.  One may wonder which singular inner functions are 
admissible as divisors of Smirnov WEP functions, in analogy to the study begun in \cite{Bo1}, \cite{BNT}.
It turns out that there aren't any besides the constants.

%For any $H_1\in QB_+(\D)$, there exists $H_2\in QB_+(\D)$ such that $|I(z)|\geq e^{-H_2(z)}$ for any $z\in\mathbb{D}$ such that $\rho(z,\Lambda)\geq e^{-H_1(z)}$.

\begin{theorem}
\label{smirnovwep}
Let $I=BS$ be an inner function, where $B$ is a Blaschke product
and $S$ is singular inner. Then $I$ satisfies the Smirnov WEP if and only if $S=e^{i\theta}$ (a unimodular constant)
and $B^{-1}\{0\}$
is a finite union of Smirnov interpolating sequences.

As a consequence, if $I$ is an inner function, the Corona Property holds for $\N^+_I$
if and only if $I$ is a Blaschke product with its zero set being a  finite union of Smirnov interpolating sequences.
\end{theorem}

\begin{proof}
By Lemma \ref{far}, there exists $H_1 \in QB_+(\D)$ large enough so that in every Whitney square $Q_{n,k}$
there exists $z_{n,k}$ such that $\rho(z_{n,k}, B^{-1}\{0\}) \ge \exp(-10 H_1(z_{n,k}))$. Applying
the Smirnov WEP property, there exists $H_2 \in QB_+(\D)$ so that 
$$
\forall z: \rho(z, B^{-1}\{0\}) \ge e^{-10 H_1(z)}, \quad
-\log |S(z)| \le -\log |I(z)| \le H_2(z).
$$
Since this is true for each $z_{n,k}$ and since $-\log|S|$ and $H_2\in \Harmd$, Harnack's inequality
implies that $-\log |S(z)| \le C H_2(z)$ on each $Q_{n,k}$, therefore on the whole disk.
But $H_2 \in QB(\D)$, so this forces $-\log |S(z)|=0$ for all $z$.

To see the second part of the theorem, suppose that the Corona Property holds for $\N^+_I$.
Then we can run the proof of (a) $\Rightarrow$ (b) for the function $I$ instead $B$. The first part
of the theorem shows that then $I$ is a Blaschke product, and by the remarks
before Theorem \ref{smirnovwep}, Proposition \ref{smirnovprop}
applies to it to yield (c).  The converse follows from Proposition \ref{smirnovprop}.
\end{proof}

\section{Proof of Theorem C and the Corollary}
\label{proofc}
 
(a)$\Rightarrow$(b). This implication follows immediately from the characterization \eqref{intN2} of Nevanlinna interpolating Blaschke sequences. 
 
(b)$\Rightarrow$(a). This is modelled after the analogue in the $H^\infty$ case (see \cite[Proposition 1]{Bo1}) and it follows a scheme very similar to the proof in the previous section: the main step is to see that $\Lambda$ is the finite union of $H$-separated sequences, for some $H\in\Harmd$; then one sees that the hypothesis implies that each of this $H$-separated sequences is Nevanlinna interpolating.

Thus let us see first that 
\begin{equation}\label{finite-weakly}
 \sup_{z\in\D}\# (\Lambda\cap D_{4H_1}(z))\leq N\ .
\end{equation}
Recall that, according to \cite[Lemma 2.4]{HMN2}, this is equivalent to $\Lambda$ being the union of $N$ $H$-separated sequences, for some $H\in\Harmd$ (actually a multiple of $H_1)$.

Fix $z\in\D$ and let $M:=\# \Lambda\cap D_{H_1}(z)$. Considering only the Blaschke factors corresponding to the zeros of $B$ in $D_{H_1}(z)$ we get the estimate
\[
 \max_{D_{H_1}(z)} |B|\leq e^{-MH_1(z)}\ .
\]
On the other hand, by an area estimate on $D_{H_1}(z)$ we get
\[
  \max_{w\in D_{H_1}(z)} \rho(w,\Lambda)\geq \frac 1{\sqrt{M}} e^{-H_1(z)}.
\]
By the hypothesis and Harnack's inequality
\begin{align*}
  e^{-MH_1(z)}\geq \max_{w\in D_{H_1}(z)} e^{-H_1(w)} \rho^N(w,\Lambda)\geq e^{-2H_1(z)}\left(\frac{e^{-H_1(z)}}{\sqrt{M}}\right)^N
  % =e^{-(N+2)H_1(z)-\frac N2\log M},
\end{align*}
 that is
 \[
  (N+2)H_1(z)+\frac N2 \log M\geq M H_1(z)\ .
 \]
There is no restriction in assuming that $H_1(z)\geq\max(N,100)$ for all $z\in\D$. Then this estimate implies
\[
 M H_1(z)\leq (N+2) H_1(z) + H_1(z)\log M, 
\]
that is
\begin{equation}\label{MN}
 M-\log M\leq N+2.
\end{equation}

In order to reach a contradiction assume that \eqref{finite-weakly} does not hold:  there exists $z\in\D$ such that 
\[
 \# \Lambda\cap D_{4H_1}(z) >N\ .
\]
Considering only the Blaschke factors in $D_{4H_1}(z)$ we get, for $w\in\partial D_{3H_1}(z)$,
\[
 |B(w)|\leq e^{-3 H_1(z) (N+1)}.
\]
By the hypothesis and Harnack then,
\begin{align*}
 \rho(w,\Lambda)\leq \left[e^{H_1(w)} |B(w)|\right]^{1/N}\leq\left[e^{2H_1(z)}e^{-3(N+1) H_1(z) }\right]^{1/N}=e^{-(3+1/N) H_1(z)}.
\end{align*}
The total pseudohyperbolic length of $\partial D_{3H_1}(z)$ is approximately $e^{-3 H_1(z)}$, hence the number of points in $\Lambda$ at distance smaller than $e^{-(3+1/N) H_1(z)}$  from $\partial D_{3H_1}(z)$ is at least approximately
\[
 \frac{e^{-3 H_1(z)}}{e^{-(3+1/N) H_1(z)}}=e^{\frac{1}N H_1(z)}.
\]
Therefore, in particular, 
\[
 M\geq e^{\frac {1}N H_1(z)}\ .
\]
This contradicts \eqref{MN} as soon as $H_1$ is assumed to be bounded below by a suitable constant depending only on $N$ ($\inf_z H_1(z)\geq N^2$ will do).

By \eqref{finite-weakly} and \cite[Lemma 2.4]{HMN2} there exists $H\in\Harmd$ (actually a multiple of $H_1$) such that $\Lambda=\cup_{j=1}^N\Lambda_j$, where each $\Lambda_j$ is $H$-separated. We shall be done as soon as we see that each $\Lambda_j$ is Nevanlinna interpolating. Let's still call $H_1$ the positive harmonic function such that the disks $D_{H_1}(\lambda)$, $\lambda\in\Lambda_j$, are pairwise disjoint. In order to check condition \eqref{intN} for $\Lambda_j$ let $B_j$ denote the Blaschke product associated to $\Lambda_j$ and fix $\lambda\in\Lambda_j$. Then 
\[
 \#[\Lambda\cap (D_{4H_1}(\lambda)\setminus D_{8H_1}(\lambda))]< N,
\]
and there exists $\beta\in(4,8)$ such that
\[
 \inf_{w\in\partial D_{\beta H_1}(\lambda)} \rho(w,\Lambda)\geq e^{-10H_1(\lambda)}\ .
\]
From here we finish as in the previous section. By hypothesis and by Harnack's inequality, for $w\in\partial D_{\beta H_1}(\lambda)$,
\begin{align*}
 |B_j(w)|\geq |B(w)|\geq e^{-H_1(w)}\rho^N(w,\Lambda)\geq e^{-2H_1(\lambda)}e^{-10NH_1(\lambda)}= e^{-(10N+2)H_1(\lambda)}\ .
\end{align*}
The function $g:=(B_j)_\lambda=B_j/b_\lambda$
is holomorphic and non-vanishing on $D_{4H_1}(\lambda)$, hence also on $D_{\beta H_1}(\lambda)$. By the minimum principle
\[
 |g(\lambda)|=(1-|\lambda|^2)|B_j^\prime(\lambda)|\geq\min_{w\in\partial D_{\beta H_1}(\lambda)}\frac{|B_j(w)|}{|b_\lambda(w)|}\geq e^{-(10N+2)H_1(\lambda)},
\]
which is condition \eqref{intN} for $\Lambda_j$, with $H=(10N+2)H_1\in\Harmd$.

\bigskip

(a)$\Rightarrow$(c) Let $\Lambda=\cup_{j=1}^N \Lambda_j$, with $\Lambda_j$ Nevanlinna interpolating for all $j=1,\dots, N$. By \eqref{intN2} there exists $H\in\Harmd$ such that
\begin{equation}\label{intNj}
 |B_j(z)|\geq e^{-H(z)} \rho(z,\Lambda_j), \quad z\in\D.
\end{equation}
Here $B_j$ denotes the Blaschke product associated to $\Lambda_j$.
In particular, there exists $H\in\Harmd$ such that $\# \Lambda_j\cap D_H(z)\leq 1$ for all $j\leq N$ and all $z\in\D$ (and therefore $\# \Lambda\cap D_H(z)\leq N$).

Fix $z\in\D$ and denote by $\lambda_1,\dots,\lambda_N$  the $N$  points of $\Lambda$ nearest to $z$, arranged by increasing distance. Let $\Lambda\cap D_H(z)=\{\lambda_1,\dots,\lambda_K\}$. Notice that $K\leq N$ and that each $\lambda_j$ belongs to a different $\Lambda_j$. Notice also that in case $K<N$ the $(K+1)$-th nearest point to $z$ is at distance no less than $e^{-H(z)}$.  Therefore, \eqref{intNj} yields
\begin{align*}
 |B(z)|&= |B_1(z)| \cdots |B_N(z)|\geq e^{-2N H(z)} \rho(z,\lambda_1)\cdots\rho(z,\lambda_K).
 %\geq e^{-N H(z)} \rho(z,\Lambda_1)\cdots\rho(z,\Lambda_N).
\end{align*}

Since
\[
 |B(N)(z)|=\frac{|B(z)|}{\rho(z,\lambda_1)\cdots \rho(z,\lambda_N) }
\]
we are done.

\bigskip

(c)$\Rightarrow$(b). Immediately, using the same notation as before,
\begin{align*}
 |B(z)|=|B(N)(z)|\rho(z,\lambda_1)\cdots \rho(z,\lambda_N)\geq e^{-H(z)} \rho^N(z,\Lambda)\ .
\end{align*}

\bigskip

[(a),(b)]$\Rightarrow$(d). Assume that $\Lambda=\cup_{j=1}^N \Lambda_j$, with $\Lambda_j$ Nevanlinna interpolating for all $j=1,\dots, N$, and that $H\in\Harmd$ is so that $\sup_{z\in\D}\# \Lambda_j\cap D_H(z)\leq 1$ for all $j=1,\dots, N$. Let $z\in\D$ and separate in two cases:

If $\rho(z,\Lambda)\geq e^{-H(z)}$ then, by (b) 
\[
 D_N(B)(z)\geq |B(z)|\geq e^{-H_1(z)} e^{-NH(z)},
\]
and we can take $H_3=H_1+NH$.

If $\rho(z,\Lambda)< e^{-H(z)}$ let $\Lambda\cap D_H(z)=\{\lambda_1,\dots,\lambda_K\}$. Notice that each $\lambda_j$ belongs to a different $\Lambda_j$ and that $K\leq N$. Consider the pseudohyperbolic divided differences defined inductively as 
\begin{align*}
 \Delta^0 B(w)&= B(w), \\
 \Delta^1 B(\lambda_1,w)&=\frac{\Delta^0 B(w)-\Delta^0 B(\lambda)}{b_{\lambda_1}(w)}=\frac{B(w)}{b_{\lambda_1}(w)}, \\
 \cdots&\cdots \\
 \Delta^K B(\lambda_1,\dots,\lambda_K,w)&=\frac{\Delta^{K-1}B (\lambda_2,\dots,\lambda_K,w)-\Delta^{K-1}B (\lambda_1,\dots,\lambda_K)}{b_{\lambda_1}(w)}\\
  &=\frac{B(w)}{b_{\lambda_1}(w)\cdots b_{\lambda_N}(w)}.
\end{align*}

Notice that $\Delta^K B(\lambda_1,\dots,\lambda_K,w)$ is holomorphic and non-vanishing in $D_H(z)$. Repeating the argument given in the implication  (b)$\Rightarrow$(a), there exists $\beta\in (2,10)$ such that $\rho(w,\Lambda)\geq e^{-10 H(z)}$, for  $w\in \partial D_{\beta H}(z)$. Then, for such $w$,
\begin{align*}
|\Delta^K B(\lambda_1,\dots,\lambda_K,w)|&=|B(N)(w)| \rho (w,\lambda_{K+1})\cdots  \rho (w,\lambda_{N})\\
&\geq  e^{-H_2(w)} e^{-2NH(z)}\geq e^{-2 H_2(z) - 2NH(z)}\ .
\end{align*}
Letting $H_3=2H_2+2 N H$ we get, by the minimum principle and Harnack,
\[
 |\Delta^K B(\lambda_1,\dots,\lambda_K,w)|\geq e^{-2H_3(z)}, \quad w\in D_{\beta H}(z).
\]
Since all these divided differences of order $K$ are uniformly bounded below (regardless of the proximity of $\lambda_1,\dots,\lambda_K$ to $z$) we get the the same bound for the pseudohyperbolic $K$-th derivative, hence
\[
 D_N(B)(z)\geq (1-|z|^2)^K |B^{(K)}(z)|=|\Delta^K B(z,\dots,z,z)|\geq e^{-2H_3(z)}.
\]

\bigskip

(d)$\Rightarrow$(c).  We first observe the following fact. 
%Fix $z\in\D$ and let $\lambda_1,\dots,\lambda_N$ be the $N$ points of $\Lambda$ nearest to $z$, arranged by increasing distance.  By hypothesis there exists $j\in\{1,\dots, N\}$ such that $(1-|z|^2)^j |B^{(j)}(z)|\geq e^{-H_3(z)}/N$, so
%\[
% |\Delta^j B(\lambda_1,\dots,\lambda_j,w)|\geq e^{-4H_3(z)}, \quad w\in D_{2 H_3}(z).
%\]
%Then
%\begin{align*}
% |B(z)|&=|\Delta^j B(\lambda_1,\dots,\lambda_j,z)| \rho(z,\lambda_1)\cdots\rho(z,\lambda_j)\geq e^{-4H_3(z)} \rho^N(z, \Lambda),
%\end{align*}
%as desired.

{\bf Claim.} Under the hypothesis (d), there exists a constant $C_1>0$ such that for any 
$z\in\D$, $\# \Lambda \cap D_{H_3+C_1}(z) \le N$.

{\it Proof of the Claim.} For any $z\in\D$, let $\lambda_1,\dots,\lambda_N, \lambda_{N+1}$ be the $N+1$ points of $\Lambda$ nearest to $z$, arranged by increasing pseudohyperbolic distance. Let 
$$
B_{N+1}(\zeta) := \prod_{m=1}^{N+1} b_{\lambda_m}(\zeta).
$$
Then there are combinatorial constants $c_i$ such that 
\[
B_{N+1}^{(l)}(\zeta) = \sum_{i\in \Z_+^{N+1} : i_1+\cdots+i_{N+1} = l} 
c_i \prod_{m=1}^{N+1} b_{\lambda_m}^{(i_m)}(\zeta).
\]
By the Cauchy estimates, $(1-|\zeta|)^k |b_\lambda^{(k)}(\zeta)|\le C_k$ for any $k\geq 1$, so that
\[
(1-|\zeta|)^l \bigl| B_{N+1}^{(l)}(\zeta) \bigr| \le C
\sum_{
\begin{subarray}
 A A \subset \{1,\dots,N+1\}\\
 \# A= N+1-l
\end{subarray}
}
\left| \prod_{m\in A} b_{\lambda_m}(\zeta)\right|,
\]
thus for any $l\le N$,
\begin{equation}
\label{locest}
(1-|z|)^l \bigl| B_{N+1}^{(l)}(z) \bigr| \le C \max_{1\le m \le N+1} \rho (z, \lambda_m).
\end{equation}
On the other hand, $B= B_{N+1} B(N+1)$, so that for $0\le j \le N$,
$$
(1-|z|)^j B^{(j)} (z) 
= \sum_{l=0}^j \begin{pmatrix} j\\l \end{pmatrix} (1-|z|)^l B_{N+1}^{(l)}(z)
(1-|z|)^{j-l} B(N+1)^{(j-l)}(z).
$$
Since $(1-|z|)^{j-l} B(N+1)^{(j-l)}(z)$ are all bounded uniformly, 
\[
(1-|z|)^j \left|B^{(j)} (z) \right| \le C \max_{0\le l \le j} (1-|z|)^l \bigl| B_{N+1}^{(l)}(z)\bigr| \le C   \max_{1\le m \le N+1} \rho (z, \lambda_m),
\]
by \eqref{locest}. If the Claim fails, there are at least $N+1$ points of $\Lambda$ in $D_{H_3 + C_1} (z)$. 
%  it must do so for a discrete sequence of $z$ tending to the unit circle, for which we have 
Summing the above inequalities for $0\le l \le N$, we deduce 
\[
D_N(B)(z)  \le C e^{-H_3(z)-C_1 } < e^{-H_3(z) }
\]
if $C_1$ is chosen sufficiently large, which contradicts the hypothesis and proves the Claim.
\qed

Let $H_4 := 8(N+1) H_3 +C_0$, where $C_0$ is to be chosen. Again to get a contradiction, we suppose that (c) does not hold, namely
$ |B(N)(z)| \le \exp(-H_4(z))$, for a discrete sequence of $z$
going to the unit circle. By the above Claim, $-\log |B(N)|$ is a positive harmonic function on the disc $D_{H_3+C_1}(z)$. 
By Harnack's inequality, for $\zeta \in D_{H_3+C_2}(z)$, with $C_2=C_1+\log 3$,
\[
\log\frac1{|B(N)(\zeta)|} \ge \frac12 \log\frac1{|B(N)(z)|} \ge \frac12 H_4(z) \ge \frac14 H_4(\zeta),
\]
so $|B(N)(\zeta)| \le \exp(-H_4(\zeta))$ on that disc.  Applying Cauchy's estimates at a point
of $D_{H_3+C_3}(z)$ for $C_3$ large enough, we get, for $0\le k \le N$,
\[
|B(N)^{(k)}(\zeta)| \le C_N \exp\left( 2 k H_3(\zeta) - \frac14 H_4 (\zeta) \right) 
\le \exp\left( -2  H_3(\zeta) - \frac14 C_0 \right) .
\]
Now we can estimate the derivatives of $B$ by writing $B= B_{N} B(N)$ and compute as above,
and for $C_0$ large enough we get $D_N(B)(z)  \le \exp\left( -2  H_3(\zeta) \right)$, 
a contradiction.

\begin{remark*}
Finite unions of Nevanlinna interpolating sequences can also be characterized through divided differences of values on $\Lambda$. 
For any $N\geq 1 $, denote 
\[
\Lambda^N=\{(\lambda_1,\ldots,\lambda_N)\in
\Lambda\times\stackrel{\stackrel{N}{\smile}}{\cdots}\times \Lambda\; :\; 
\lambda_j\not=\lambda_k\ \textrm{if}\  j\not=k\},
\] 
and consider the set $X^{N-1}(\Lambda)$ consisting of the
functions $\omega$ defined in $\Lambda$ with divided differences of order $N-1$ uniformly
controlled by a positive harmonic function $H$ i.e., such that for some $H\in\Harmd$,
\[
\sup_{(\lambda_1,\ldots,\lambda_N)\in \Lambda^N} \vert
\Delta^{N-1}\omega(\lambda_1,\ldots,\lambda_N)\vert
e^{-[H(\lambda_1)+\cdots+H(\lambda_N)]}<+\infty\ .
\]

According to the Main Theorem in \cite{HMN2}, $\Lambda$ is the union of $N$ Nevanlinna interpolating sequences if and only if the trace $\N|\Lambda$ coincides with $ X^{N-1}(\Lambda)$.
Again, in accordance to our principle, this is the analogue for $\N$ of Vasyunin's result for $H^\infty$ (see \cite{Vas83}, \cite{Vas84}). See \cite{BNO} for the analogue result in Hardy spaces. 
\end{remark*}

\begin{proof}[Proof of the Corollary]
We shall see first that the zeros of $B-g$ are just a small perturbation of $\Lambda$. By the stability of Nevanlinna interpolating sequences, we shall then deduce that these zeros can also be split into $N$ Nevanlinna interpolating sequences.

Assume that $\Lambda=\cup_{j=1}^N \Lambda_j$, where each $\Lambda_j$ is Nevanlinna interpolating. By Theorem C, there exists $H_1\in\Harmd$ such that
\[
 \# \Lambda_j\cap D_{H_1}(z)\leq 1, \quad z\in\D, \quad j=1,\dots, N, 
\]
and therefore 
\[
 \# \Lambda \cap D_{H_1}(z)\leq N, \quad z\in\D . 
 % , \quad j=1,\dots, N. 
\]
There is no restriction in assuming that Theorem C (b) holds for this same $H_1$.

Then, for any $\lambda\in\Lambda$ there exists $k\in\{1,\dots, N+1\}$ such that
\[
 \rho(\zeta,\Lambda)\geq e^{-(N+1) H_1(\zeta)}, \quad\textrm{for $\zeta\in\partial D_{kH_1}(\lambda)$}.
\]
Then, by  the hypothesis and (b) in Theorem C, for such $\zeta$,
\begin{align*}
 & |B(\zeta)|\geq e^{-H_1(\zeta)}\rho^N (\zeta,\Lambda)\geq e^{-H_1(\zeta)} e^{-(N+1)N H_1(\zeta)}, \\
 & |(B-g)(\zeta)- B(\zeta)|=|g(\zeta)|\leq e^{-H_0(\zeta)}.
\end{align*}
Choosing $H_0>(N^2 + N + 1) H_1$ we can apply Rouch\'e's Theorem to deduce that
\[
 \# Z(B-g)\cap D_{kH}(\lambda)=\# \Lambda \cap D_{kH}(\lambda).
\]
This shows that for each $\lambda\in\Lambda$ there is a zero of $B_1$ at a distance smaller than $e^{-H_1 (\lambda)}$.  To see that $B-g$ has no other zeros consider any contour $\Gamma$ in $\D$ for which
\[
 \rho(\zeta,\Lambda)\geq e^{-H_1(\zeta)}, \quad \zeta\in \Gamma.
\]
The estimate above and Rouch\'e's Theorem show now that $B$ and $B-g$ have the same number of zeros in the interior of $\Gamma$. Since  $B-g$ has no other zeros we can split $Z(B-g)$ into $N$ sequences which, by the stability under such perturbations of Nevanlinna interpolating sequences \cite[Corollary 2.3]{HMN1}, are also Nevanlinna interpolating.

We have then the factorization $B-g=B_1 G$,  where $B_1$ is a finite product of NIBP's and $G$ is invertible in $\N$. Then,  for $\zeta\in\partial\D$,
\[
 \frac 1{|G(\zeta)|}=\frac{|B_1(\zeta)|}{|B(\zeta)-g(\zeta)|}\leq \frac 1{1-e^{-H_0(\zeta)}},
\]
and therefore $1/G\in H^\infty$.
\end{proof}

\end{document}